\documentclass[a4paper,12pt]{amsart}
\usepackage{enumerate}
\usepackage{amssymb,latexsym,amsxtra,amscd}
\usepackage[mathscr]{eucal}
\usepackage{latexsym}
\usepackage{amsfonts}
\usepackage{amssymb}
\footskip=10mm  \numberwithin{equation}{section}

\newcounter{obr}[section]

\newcounter{pvv}[section]
\renewcommand{\thepvv}{\thesection.\arabic{pvv}}
\newenvironment{pv}[2][]{\begin{trivlist}\refstepcounter{pvv}%
\item[\hspace{\labelsep}\normalfont\bfseries\thepvv. #2%
  \def\tmp{#1}\ifx\tmp\empty\else{} (#1)\fi.]}%
{\end{trivlist}}

\newenvironment{theo}[1][]{\begin{pv}[#1]{Theorem}\begin{itshape}}{\end{itshape}\end{pv}}
\newenvironment{pr}{\begin{pv}{Proposition}\begin{itshape}}{\end{itshape}\end{pv}}
\newenvironment{lem}{\begin{pv}{Lemma}\begin{itshape}}{\end{itshape}\end{pv}}

\newcommand{\cc}{\ensuremath{\mathbb C}}

\newcommand{\8}{\ensuremath{\infty}}

\newcommand{\f}{\ensuremath{\varphi}}

\newcommand{\de}{\ensuremath{\delta}}
\newcommand{\psai}{\ensuremath{\psi}}

\newcommand{\si}{\ensuremath{\sigma}}

\newcommand{\ifff}{if, and only if,}
\newcommand{\vNa}{von Neumann algebra}

\newcommand{\Ca}{$C^\ast$-algebra}

\newcommand{\HS}{Hilbert space}

\newcommand{\st}{such that}

\newcommand{\set}[2]{\ensuremath{\{ #1\,|\; #2\}}}

\newcommand{\ac}{absolutely continuous with respect to}
\newcommand{\uac}{uniformly absolutely continuous with respect to}
\newcommand{\VHS}{Vitali-Hahn-Saks Theorem}

\begin{document}
\title{ Noncommutative Vitali-Hahn-Saks Theorem holds precisely for finite $W^\ast$-algebras.
}

\author{E.Chetcuti and J.Hamhalter}

\begin{abstract}It is shown that the bona fide
generalization of the Vitali-Hahn-Saks Theorem to von Neumann
algebras is possible if, and only if, the algebra is finite.  This
settles the problem on the noncommutative Vitali-Hahn-Saks Theorem completely and
provides new means of characterizing finite von Neumann algebras.
\end{abstract}
\maketitle

%{\small Abstract:  In this paper we show that the direct generalization
%of the \VHS{} holds for all locally convex space valued measures on
%finite \vNa s. On the other hand, we prove that the \VHS{} fails for
%all infinite \vNa s even if one restricts to scalar measures. This
%settles the problem on \VHS{} for \vNa s completely and provides  a
%new characterization of finite \vNa s. }

\section{Introduction and Preliminaries}

The Vitali-Hahn-Saks Theorem is one of the fundamental results of measure theory.  Let $K$ be a set of (scalar-valued) measures on a $\sigma$-field
$\mathcal A$ of measurable sets.  Suppose that $K$ is  relatively
compact in the topology of pointwise convergence on the elements of
$\mathcal A$ and that every $\f\in K$ is \ac{} some fixed positive
measure $\psi$ on $\mathcal A$. In its classical form, the \VHS{}
asserts that $K$ is \uac{} $\psi$. 

The generalization of the \VHS{} to \vNa s has recently received a
great deal of attention. Deep classical results of Aarnes and
Akemann \cite{Aarnes,Akemann} have been considerably extended to \Ca
s and \vNa s by Brooks, Sait\^o and Wright in a series of remarkable
papers \cite{B+W,BrWrS,B+S+W}. However, in order to obtain uniform
absolute continuity, in these papers it is assumed that $K$ is
pointwise {\em strongly} absolutely continuous with respect to
$\psi$. Let us recall that a normal functional \f{} on a \vNa{} $M$
is \ac{} a normal positive functional  $\psi$ on $M$ if $\f(p)=0$
whenever $p$ is a projection in $M$ with $\psi(p)=0$. On the other
hand, \f{} is said to be strongly \ac{} $\psi$ if both the absolute
value $|\f_h|$ of the hermitian part of \f{} and the absolute value
$|\f_{ah}|$ of the antihermitian part of $\f$ are \ac{} \psai.
Strong absolute continuity is much more stringent than absolute
continuity; particularly if we consider vector-valued measures. The
two notions coincide if the algebra is abelian or when all measures
concerned are positive.
 Therefore a natural question has erupted on whether
  the \VHS{} holds without assuming strong absolute continuity in the hypothesis. In
  \cite{ChetHam} we showed  the 'genuine' form of the \VHS{} can
be obtained provided the control measure is faithful. In this paper
we settle completely the status of the \VHS{} for \vNa s. We show
that the direct generalization is possible for vector-valued
measures on finite \vNa s. On the other hand, we prove that if the
algebra is infinite,  the \VHS{} fails, even if
we restrict to scalar-valued measures. This new measure-theoretic
characterization of finite \vNa s  complements the wide range of
hitherto known topological, functional-analytic, and
lattice-theoretic characterizations of finiteness in the Murray-von
Neumann comparison theory (see e.g. \cite{Hamhalter,Saito,
Takesaki}). Since for positive measures the \VHS{} holds
irrespectively of the von Neumann algebra (see e.g. \cite[Theorem 4.6]{ChetHam}), the results obtained in
this paper exhibit the delicate interplay that exists between the
measure and its absolute value in the noncommutative situation. Pursuing
this matter further, we show that for each infinite \vNa{} there is
a weakly relatively compact subset $K$ in the predual \st{} the
\VHS{} holds for $K$, but not for the set $|K|$ (where $|K|$ denotes
the set of absolute values of the functionals of $K$). This extends
a classical result of Sait\^{o} \cite{Saito}, who proved that a
\vNa{} $M$ is finite \ifff{} the following condition holds: A subset
$K$ of the predual of $M$ is weakly relatively compact exactly when
$|K|$ is weakly relatively compact. In this connection we  have
obtained further descriptions of finite \vNa s.\par

Let us recall basic facts and fix the notation. Our standard
reference for operator algebras is \cite{Takesaki}. For a normed
space $F$ we shall  use the symbol $F_1$ for its closed unit ball.
The symbol $B(\mathcal H)$ will be reserved for the algebra of all
bounded operators acting on a \HS{} $\mathcal H$. Throughout the
paper, $M$ will stand for a \vNa. The symbols $M_\ast$ and
$M_\ast^+$ shall denote the predual and the positive part of the
predual, respectively. If $M$ acts on a \HS{} $\mathcal H$ and
$\eta,\xi\in \mathcal H$, we shall denote by $\omega_{\eta,\xi}$ the
linear functional $\omega_{\eta,\xi}(x)=(x\eta, \xi)$ ($x\in M$). We
also set $\omega_\eta=\omega_{\eta, \eta}$. If $\f\in M_\ast$ we
shall denote by $|\f|$ the absolute value of the functional \f{}¨
(see e.g. \cite{Takesaki} for more details). Let $P(M)$ denote the
projection lattice of $M$. Two projections $p,q\in P(M)$ are called
orthogonal if $pq=0$.

We denote by $\sigma(M,M_\ast)$ the weak$^\ast$-topology on $M$,
i.e. the weakest topology compatible with the duality $\langle
M_\ast,M\rangle$.  The strongest topology on $M$ compatible with
this duality (the Mackey topology) is denoted by $\tau(M,M_\ast)$.
We recall that the $\tau(M,M_\ast)$ topology on $M$ coincides with
the topology of uniform convergence on weakly relatively compact
subsets of $M_\ast$.  Lying between these topologies we have the
$\si$-strong topology  $s(M,M_\ast)$ determined by the family of seminorms
$\set{\rho_{\omega}}{\omega\in M_\ast^+}$ where
$\rho_{\omega}(x)=\omega(x^\ast x)^{1/2}$; and the $\si$-strong$^\ast$
topology $s^\ast(M,M_\ast)$ determined by the family of the
seminorms $\set{\rho_\omega,\rho^\ast_\omega}{\omega\in M_\ast^+}$,
where $\rho^\ast_{\omega}(x)=\omega( x x^\ast)^{1/2}$. Let us recall
that on bounded parts of $M$ the $\si$-strong$^\ast$ topology coincides
with the Mackey topology.

Let $X$ be a locally convex vector space. A {\em completely
additive} measure $\mu:P(M)\to X$   is a map satisfying
$\mu\biggl(\sum_{p\in \Gamma} p\biggr) =\sum_{p\in \Gamma}\mu(p)\,,$
whenever $\Gamma$ is a set of pairwise orthogonal projections. By
the symbol $B_{ca}(M,X)$ we shall denote the set of all bounded
linear maps of $M$ into $X$ which restrict to completely additive
measures on $P(M)$. (If $X=\cc$, then $B_{ca}(M,X)$ reduces to the
predual of $M$.) By the vector Gleason theorem \cite{BunceWright}
there is a one-to-one correspondence between bounded completely
additive measures on $P(M)$  and operators in $B_{ca}(M,X)$ (see
\cite{ChetHam,QMT} for more details).

A subset $K\subset B_{ca}(M,X)$ is said to be \emph{pointwise
absolutely continuous} with respect to $\psi\in M_\ast^+$ (in
symbols $K\ll_p\psi$) if for every $T\in K$ and neighbourhood $U$ of
$0\in X$, there is a $\delta>0$ such that $Tp\in U$ whenever $p\in
P(M)$ and $\psi(p)<\delta$.  For every $T\in B_{ca}(M,X)$ we write
$T\ll \psi$ if $\{T\}\ll_p \psi$.  A subset $K\subset B_{ca}(M,X)$
is said to be \emph{uniformly absolutely continuous} with respect to
$\psi\in M_\ast^+$ (in symbols $K\ll_u\psi$) if for every
neighbourhood $U$ of $0\in X$, there is a $\delta>0$ such that
$Tp\in U$ for every $T\in K$ whenever $p\in P(M)$ and
$\psi(p)<\delta$. (For more details on the interrelations between
various concepts of absolute continuity see \cite{ChetHam}.)

Given a subset $K\subset M_\ast$ we define $K_p=\set{\psi\in
M_\ast^+}{K\ll_p \psi}$ and $K_u=\set{\psi\in M_\ast^+}{K\ll_u
\psi}$.  Of course, $K_u\subset K_p$. A subset $K\subset M_\ast$ is
said to have the {\em Vitali-Hahn-Saks property} (VHS-property in short) if
$K_p\not=\emptyset$ and $K_p=K_u$. A  deep result of Akemann states
that $K_u$ is nonempty if $K\subset M_\ast$ is weakly relatively
compact \cite{Akemann}. If $K\subset M_\ast$ is bounded and enjoys
the Vitali-Hahn-Saks property, then it has a control measure and
therefore $K$  is weakly relatively compact. However, it will follow
from our discussion that in the predual of any infinite algebra
there are weakly relatively compact subsets that do not have the
Vitali-Hahn-Saks property.

\section{Results}

Before giving the proof of the \VHS{} for finite algebras we recall
that the following three conditions are equivalent: (i)~$M$ is
finite; (ii)~The $\ast$-operation is \si-strongly continuous;
(iii)~On bounded parts of $M$ the ($\si$-) strong topology  agrees with the
Mackey topology \cite[p. 333, Exercise 5]{Takesaki}.
%Let $\tau(M,M_\ast)$ denote the Mackey
%topology on $M$, i.e. the topology of uniform convergence on weakly
%relatively compact subsets of $M_\ast$.
%The Mackey topology
%coincides with the $\si$-strong$^\ast$ topology on bounded parts of
%$M$. Moreover, \\

\begin{theo}\label{finite}
Let $M$ be a finite \vNa{} and $K\subset B_{ca}(M,X)$ a relatively
compact set in the topology of pointwise convergence on elements of
$M$. Let $\psi\in M_\ast^+$. If $K$ is pointwise absolutely continuous with respect to $\psi$, then $K$
is uniformly absolutely continuous with respect to $\psi$.
\end{theo}
\begin{proof}
First we shall prove this theorem for the scalar case, i.e. when
$X=\cc$. In this case $K$ is a weakly relatively compact subset of
$M_\ast$. Let $(e_k)$ be a sequence of projections \st{}
$\psi(e_k)\to 0$.  Denote by $s(\psi)$  the support projection of
$\psi$ and let $N=s(\psi)M s(\psi)$.  Since $\psi$ is faithful on
$N$, $s(\psi) e_k s(\psi)\to 0$ in the $s(N,N_\ast)$ topology (and
therefore in the $\sigma(N,N_\ast)$ topology) \cite[p. 148, Prop.
5.3]{Takesaki}. On $N$ the $\sigma(N,N_\ast)$ topology coincides
with the relativized $\sigma(M,M_\ast)$ topology.  Thus it follows
that $e_k s(\psi)\to 0$ in the $s(M_\ast,M)$ topology.  Since $M$ is
finite $s(\psi)e_k=(e_ks(\psi))^\ast\to 0$ in the $s(M_\ast,M)$
topology.  This implies that
\[
s(\psi)e_ks(\psi),\,(1-s(\psi))e_ks(\psi),\,s(\psi)e_k(1-s(\psi))\to
0
\]
in the $s(M,M_\ast)$ topology - and therefore in the $\tau(M,M_\ast)$ topology.
  Moreover, since every $\f\in K$ is absolutely
continuous with respect to $\psi$ we have  $\f((1-s(\psi))\, e_k\,
(1-s(\psi))=0 $ for every $k$ \cite[Proposition 2.2]{ChetHam}.
Combining we get that  as $k\to\infty$

\begin{align*}
\sup_{\f\in K}|\f(e_k)| \le& \sup_{\f\in K} |\f(s(\psi)\, e_k\,
s(\psi))|
+ \sup_{\f\in K} |\f((1-s(\psi))\, e_k\, s(\psi))| \\
+ &\sup_{\f\in K} |\f(s(\psi)\, e_k\, (1-s(\psi)))|\to 0\,.
\end{align*}

Let us return to the general vector case. Let $U$ be a convex,
closed, circled neighbourhood of zero in $X$ with the polar
$U^0=\set{f\in X^\ast}{|f(x)|\le 1}$. Then the set
\[ K_U=\set{f\circ T}{f\in U^0,\, T\in K}\]
is a weakly relatively compact subset of $M_\ast$ by \cite[Theorem
3.3]{ChetHam}. By the previous part of the proof $K_U\ll_u\psi$.
% is
%\uac{} \psai{}.
In other words, there is a $\de>0$ \st{} $
|f(T(p))|\le 1$ for all $f\in U^0$ and $T\in K$, whenever $p\in
P(M)$ and $\psi(p)<\de$. Therefore, if $\psi(p)<\de$, then $T(p)\in
U$ for all $T\in K$ by the Bipolar Theorem.
\end{proof}

\begin{pr}\label{proposition}
Let $M$ be an infinite \vNa{}. Then there is a normal state $\psi$
on $M$ and a sequence $(\f_k)$ of normal functionals on $M$ \st{}
each $\f_k$ is \ac{} $\psi$,  $(\f_k)$ converges to zero pointwise
on $M$, but
 $(\f_k)$ is not \uac{} $\psi$.
\end{pr}

\begin{proof} Suppose that $M$ acts on a \HS{} $H$.
Since $M$ has nonzero properly infinite part, there is a sequence
$(e_n)$ of mutually orthogonal nonzero projections in $M$ which are pairwise
equivalent. Therefore, there is a sequence $(u_k)$ of partial
isometries in $M$ \st{} $ u_k^\ast\, u_k= e_1$, $ u_k\,u_k^\ast =
e_k$ for each $k\ge 2$. Let us choose a unit vector $\xi_1\in
e_1(H)$ and put $\xi_k=u_k\, \xi_1\,,\quad k\ge 2\,.$ The sequence
$(\xi_k)$ is orthonormal. Let us further define
\[ \psi=\sum_{n=2}^\8 \frac 1{2^n} \omega_{\xi_n}\quad\text{and}\quad
\f_k=\omega_{\xi_{k}, \xi_{1}}\,\,(k\ge 2). \]
 It is clear that $\f_k\to 0$ weakly.
Let $p\in P(M)$ with $\psi(p)=0$. Then, for each $k\ge 2$,
$p\xi_k=0,$ and so $ \f_{k}(p)=(p\xi_{k},\xi_{1})=0 \,.$ Whence,
the sequence $(\f_k)_{k\ge 2}$ is pointwise \ac{} \psai.
For each $k\ge 2,$ define $ h_k=\frac 12(e_1+e_k+u_k+u^\ast_k)\,.$
Observe that $h_k$ is a projection lying underneath $e_1+e_k$. We
can compute
\begin{eqnarray*}
\psi(h_k)&=&\sum_{n=2}^\8 \frac 1{2^n}\, \omega_{\xi_n}(h_k)=
\sum_{n=2}^\8 \frac 1{2^n}\, (h_k\, \xi_n,\xi_n)=\\
&=&\frac 1{2^k}(h_k\xi_k, \xi_k)\,.
\end{eqnarray*}
Thus $ \psi(h_k)\le \frac 1{2^k}\to 0 \mbox{ as } k\to\8\,.$ On the
other hand $ \f_{k}(h_k)=(h_k\xi_k,\xi_1)= \frac 12\,,$ whenever
$k\ge 2$. Whence, for each fixed $n$
\[ \sup_k |\f_k(h_n)|\ge \frac 12\,, \]
 and so the sequence $(\f_k)$ is not \uac{} $\psi$.
\end{proof}

\begin{theo}\label{characterization1}
The following conditions are equivalent
\begin{enumerate}
\item $M$ is finite
\item Let $K\subset M_\ast$ be a weakly relatively compact set and $\psi\in M_\ast^+$.
If $K\ll_p \psi$, then $K\ll_u \psi$
\end{enumerate}
\end{theo}

 Sait\^o showed  that if $M$ is a finite von Neumann
algebra and $K\subset M_{\ast}$ is weakly relatively compact, then
$|K|$ is also weakly relatively compact \cite[Theorem 1]{Saito}.
Moreover, he proved that this property of the predual characterizes
finite algebras completely.  (Of course, if $|K|$ is weakly
relatively compact, then $K$ is also weakly relatively compact
irrespective of whether the algebra is finite or not.)
Theorem~\ref{finite} and \cite[Theorem III.9]{Akemann} (see also
\cite[p.149,Theorem 5.4]{Takesaki}) tells us that for finite
algebras, weakly relatively compact sets in the predual coincide
with the bounded sets enjoying the Vitali-Hahn-Saks property.
Combining with the result of Sait\^o we conclude that for finite
algebras:
\[
K\mbox{ has the VHS-property if, and only if, }|K| \mbox{ has the
VHS-property},
\]
for every  bounded subset $K$ of the predual.  In Theorem
\ref{characterization2} below we show that for $\sigma$-finite algebras
both implications imply finiteness of the algebra.  Let us first
prove the following lemmas.

\begin{lem}\label{lemma1}
Let $M$ be a von Neumann algebra acting on a Hilbert space $\mathcal
H$ and let $(\eta_i)\subset \mathcal H$ be a $($norm$)$ convergent
sequence of vectors.  Then, the set
\[
\set{\omega_{\eta_i\,,\xi}}{i\in\mathbb N,\,\xi\in \mathcal
H_1}\subset M_\ast\] is weakly relatively  compact.
\end{lem}
\begin{proof}Let $\mathcal H^1$ and $\mathcal H^2$ be two copies of $\mathcal H$.
Equip $\mathcal H^1_1$ with the norm topology and $\mathcal H^2_1$
with the weak topology.
 Consider the Cartesian
product $\mathcal H^1_1\times\mathcal H^2_1$.  One can easily check
that the function $\Gamma:(\eta,\xi)\in \mathcal H^1_1\times\mathcal
H^2_1\mapsto \omega_{\eta,\,\xi}\in M_\ast$ is continuous (where
$M_\ast$ is equipped with the weak topology). Therefore, $\Gamma$
maps relatively compact subsets of $\mathcal H^1_1\times\mathcal
H^2_1$ into weakly relatively compact subsets of $M_\ast$. But,
observe that the set
\[
\set{(\eta_i,\xi)}{i\in\mathbb N,\, \xi\in \mathcal H^2_1}\subset
\mathcal H^1_1\times\mathcal H^2_1\] is relatively compact.
\end{proof}

\begin{lem}\label{lemma2}Let $M$ be a von Neumann algebra and let $A$ be a
von Neumann subalgebra
of $M$ sharing the same unit $\mathbf{1}$ of $M$. For any
$\varphi\in M_\ast$, denote by $\widehat{\varphi}$ the restriction
of $\varphi$ to $A$. If $\Vert
\varphi\Vert=\Vert\widehat{\varphi}\Vert$ for some $\varphi\in
M_\ast$, then $\widehat{|\varphi|}=|\widehat{\varphi}|.$
\end{lem}
\begin{proof}We observe that
\[
\Vert\,\widehat{|\varphi|}\,\Vert=\widehat{|\varphi|}(\mathbf{1})=
|\varphi|(\mathbf{1})=\Vert\,|\varphi|\,\Vert=\Vert\varphi\Vert=
\Vert\widehat{\varphi}\Vert.\] Moreover, \[
|\varphi(x)|^2\le\Vert\varphi\Vert\,|\varphi|(x\,x ^\ast) \quad
\mbox{ for all } x\in M \,,\] which implies that
$|\widehat{\varphi}(x)|^2\le\Vert\widehat{\varphi}\Vert\,\widehat{|\varphi|}(x
x^\ast)$ for all $x\in A$.  Thus result follows from \cite[p. 143,
Proposition 4.6]{Takesaki}.

\end{proof}

\begin{theo} \label{characterization2}
Let $M$ be a $\sigma$-finite von Neumann algebra.  The
following statements are equivalent:
\begin{enumerate}
\item $M$ is finite;

\item $|K|$ has the Vitali-Hahn-Saks property implies that $K$ has the
Vitali-Hahn-Saks  property, for all bounded $K\subset {M_\ast}$.

\item $K$ has the Vitali-Hahn-Saks property implies that $|K|$ has the
 Vitali-Hahn-Saks  property, for all bounded $K\subset {M_\ast}$;
\end{enumerate}
\end{theo}
\begin{proof}\
We know that  if $M$ is finite, then weakly relative compact sets of
the predual coincide with the bounded sets enjoying the
Vitali-Hahan-Saks property. Therefore (i) $\Longrightarrow$ (ii) and
(i) $\Longrightarrow$  (iii) by \cite[Theorem 1]{Saito}.

Suppose now that $M$ is not finite. Consider the sequence
$\f_k=\omega_{\xi_{k}\,,\xi_1}$, $k\ge 2$,  constructed in the proof of
Proposition~\ref{proposition}. It was shown that the sequence
$K=(\f_k)_{k\ge 2}$ does not have the Vitali-Hahn-Saks property. On the other
hand, $|K|=\{\omega_{\xi_1}\}$ and therefore it has the
Vitali-Hahn-Saks property. Whence, (ii) $\Longrightarrow$ (i).

It remains to show that (iii) $\Longrightarrow$ (i). For this, we
shall suppose that $M$ is not finite and  construct a sequence
$K\subset M_\ast$ \st{} $K$ has the Vitali-Hahn-Saks property while
$|K|$ has not the Vitali-Hahn-Saks property. Since $M$ has a nonzero
properly infinite direct summand, we can assume that $M$ is properly
infinite, and thereby isomorphic to the algebra $B(\mathcal
H)\overline{\otimes} N$ where $\mathcal H$ is a separable
infinite-dimensional Hilbert space, and $N$ is a $\sigma$-finite von
Neumann algebra.   $N$ admits a faithful normal representation
$\{\pi,\mathcal K\}$ with a separating and cyclic (unit) vector
$\eta\in\mathcal K$.  Therefore, upon replacing $M$ by
$(I\overline{\otimes}\pi)M$, we can assume that $M$ acts on
$\mathcal H\otimes \mathcal K$.

Let $(\eta_i)$ be a total and convergent sequence of unit vectors in
$\mathcal H$. For each $i\in\mathbb N$, let
$\bar\eta_i=\eta_i\otimes\eta$.  We show that the sequence
$(\bar\eta_i)$ is separating for $M$.  For this, it is enough to
show that $(\bar\eta_i)$ is cyclic for $M'$. This follows by the
commutation theorem for tensor products after observing that
\[
[M'\{\bar\eta_i\}_i]=[\mathbb{C}\overline{\otimes}
N'\{\eta_i\otimes\eta\}_i]= \mathcal H\otimes\mathcal K.
\]
Consequently, the set in the predual
\[
K=\set{\omega_{\bar\eta_i,\bar\xi}}{i\in\mathbb
N,\,\bar\xi\in(\mathcal H \otimes\mathcal K)_1}\] is separating for
$M$.  In view of Lemma~\ref{lemma1}, $K$ is also weakly relatively
compact. Since $M$ is $\sigma$-finite, $M$ admits a faithful normal
state. Thus $K_p\ne \emptyset$.  Moreover, since $K$ is separating
for $M$, every $\psi\in K_p$ is necessarily faithful. By
\cite[Theorem 4.1]{ChetHam} it follows that $K_p=K_u$, i.e. $K$
enjoys the Vitali-Hahn-Saks  property.

We show that $|K|$ does not have the VHS-property. Let $(\xi_j)$ be
an orthonormal basis of $\mathcal H$ and let $\bar\xi_j=\xi_j\otimes
\eta$.  We show that the set
\[
\set{|\omega_{\bar\eta_i,\,\bar\xi_j}|}{i,\,j\in\mathbb N}\] is not
weakly relatively compact.  If we identify $B(\mathcal H)$ with
$B(\mathcal H)\otimes \mathbf{1}_N$, we can assume that $B(\mathcal
H)$ is a von Neumann subalgebra of $M$, sharing the same unit of
$M$.  The restriction to $B(\mathcal H)$ of
$\omega_{\bar\eta_i,\,\bar\xi_j}$ is $\omega_{\eta_i,\,\xi_j}$.
Since $1=\Vert
\omega_{\eta_i,\,\xi_j}\Vert=\Vert\omega_{\bar\eta_i,\,\bar\xi_j}\Vert$,
by Lemma~\ref{lemma2}, it follows that the restriction to
$B(\mathcal H)$ of $|\omega_{\bar\eta_i,\,\bar\xi_j}|$ is
$|\omega_{\eta_i,\,\xi_j}|$.  But observe that
$|\omega_{\eta_i,\,\xi_j}|=\omega_{\xi_j}$. However, the set
$\set{\omega_{\xi_j}}{j\in\mathbb N}$ is not weakly relatively
compact.  This means that $|K|$ cannot have the Vitali-Hahn-Saks property.
\end{proof}

{\bf Acknowledgment.} The work of Jan Hamhalter was supported by the
research plans of the Ministry of Education of the Czech Republic
No. 6840770010. Jan Hamhalter also thanks to the project of the Grant Agency of
the Czech Republic no. 201/07/1051 "Algebraic and Measure-Theoretic
Aspects of Quantum Structures" and to the Alexander von Humboldt
Foundation, Bonn.

\vspace{1cm}

E.Chetcuti, Department of Mathematics, Junior College, University of Malta, Msida MSD 06, Malta\\

J.Hamhalter, Czech Technical University, Faculty of Electrical Engineering, Department of Mathematics,
Technicka 2, 166 27 Prague 6, Czech Republic. e-mail: hamhalte@math.feld.cvut.cz\\

\end{document}